\documentclass[11pt]{article}
\usepackage{enumerate}
\usepackage{amssymb,a4wide,latexsym,makeidx,epsfig,fleqn}
\usepackage{amsthm}
\usepackage{amsmath}
\usepackage{enumerate}
\newtheorem{prob}{Problem}
\newtheorem{lem}{Lemma}[section]
\newtheorem{thm}{Theorem}[section]
\newtheorem{cor}{Corollary}[section]

\newtheorem{clm}{Claim}[section]

\allowdisplaybreaks[4]
\begin{document}
\textwidth 150mm \textheight 225mm
\title{The $\alpha$-index of graphs without intersecting triangles/quadrangles as a minor
\thanks{Supported by the National Natural Science Foundation of China (No.  12271439).}}
\author{{Yanting Zhang$^{a,b}$, Ligong Wang$^{a,b,}$\thanks{Corresponding author.}}\\
\small $^a$ School of Mathematics and Statistics,\\
\small Northwestern Polytechnical University, Xi’an, Shaanxi 710129, P.R. China\\
\small $^b$ Xi’an-Budapest Joint Research Center for Combinatorics,\\ 
\small Northwestern Polytechnical University,  Xi’an, Shaanxi 710129, P.R. China\\
\small E-mail:
yantzhang@163.com, lgwangmath@163.com.}
\date{}
\maketitle
\begin{center}
\begin{minipage}{120mm}
\vskip 0.3cm
\begin{center}
{\small {\bf Abstract}}
\end{center}
{\small The $A_{\alpha}$-matrix of a graph $G$ is the convex linear combination of the adjacency matrix $A(G)$ and the diagonal matrix of vertex degrees $D(G)$, i.e., $A_{\alpha}(G) = \alpha D(G) + (1 - \alpha)A(G)$,
where $0\leq\alpha \leq1$. The $\alpha$-index of $G$ is the largest eigenvalue of $A_\alpha(G)$.  Particularly, the matrix $A_0(G)$ (resp.  $2A_{\frac{1}{2}}(G)$) is exactly the adjacency matrix (resp. signless Laplacian matrix) of $G$. He, Li and Feng [arXiv:2301.06008 (2023)] determined the extremal graphs with maximum adjacency spectral radius
among all graphs of sufficiently large order without intersecting triangles and quadrangles as a minor, respectively.
Motivated by the above results of He, Li and Feng, 
in this paper we characterize
the extremal graphs with maximum $\alpha$-index among all graphs of sufficiently large order without intersecting triangles and quadrangles as a minor for any $0<\alpha<1$, respectively. As by-products, we determine the extremal graphs with maximum signless Laplacian radius among all graphs of sufficiently large order without intersecting triangles and quadrangles as a minor, respectively.

 }
\end{minipage}
\end{center}

\section{Introduction }
\label{sec:ch6-introduction}
Let $G=(V(G),E(G))$ be an undirected simple graph.
The \textit{adjacency matrix}  of $G$ is the $n \times n$ matrix   $A(G)=(a_{ij})$, where $a_{ij}=1$
if $v_{i}$ is adjacent to $v_{j}$, and $0$ otherwise. The largest eigenvalue of $A(G)$ is called the \textit{adjacency spectral radius} of $G$.  Let $D(G)$ be the diagonal matrix of vertex degrees of $G$.
The \textit{signless Laplacian matrix} of $G$ is defined as $Q(G)=D(G)+A(G)$. The largest eigenvalue of $Q(G)$,
denoted by $q(G)$, is called the \textit{signless Laplacian spectral radius} of $G$.
For two vertex disjoint graphs $G$ and $H$, we denote by $G \cup H$ the \textit{union} of $G$ and $H$, and $G \vee H$  the \textit{join} of $G$ and $H$, i.e., joining every vertex of $G$ to every vertex of $H$. Denote by $kG$ the union of $k$ disjoint copies of $G$.
As usual, denote by $K_{n}$ (resp. $P_{n}$) the \textit{complete graph} (resp. the \textit{path}) of order $n$, and $K_{m, n}$ the \textit{complete bipartite graph} on $m + n$ vertices. For a graph $G$, let $\overline{G}$ be its \textit{complement}. Denote by $F_s$ the friendship graph obtained from $s$ triangles by sharing a common
vertex, i.e., $F_s= K_1 \vee sK_2$, where $s\geq1$. Also, denote by $Q_t$ the graph obtained from $t$ quadrangles by sharing a common vertex. Let $M_{n-t}$ be an $(n-t)$-vertex graph consisting of $\lfloor\frac{n-t}{2}\rfloor$ disjoint edges. 

Given two graphs $G$ and $H$, $H$ is a \textit{minor} of $G$
if $H$ can be obtained from a subgraph of $G$ by contracting edges.
A graph is said to be \textit{$H$-minor-free} if it does not contain $H$ as a minor. Minors play a key role in graph theory, and extremal problems on forbidding
minors have attracted appreciable amount of interests.
For example, it is useful for
studying the structures and properties of graphs. Wagner \cite{Wagner} showed that a  graph is planar if and only if it is $\{K_{3,3}, K_5\}$-minor-free, while Ding and Oporowski \cite{Ding}  showed that a graph is outer-planar if and only if it is  $\{K_{2,3}, K_4\}$-minor-free.

The $A_{\alpha}$-matrix of a graph $G$ is defined by Nikiforov \cite{Nikiforov} as the convex linear combination of the adjacency matrix $A(G)$ and the diagonal matrix of vertex degrees $D(G)$, i.e.,
$$A_{\alpha}(G) = \alpha D(G) + (1 - \alpha)A(G),$$
where $0\leq\alpha \leq1$.
The \textit{$\alpha$-index} (or \textit{$A_\alpha$-spectral radius} ) of $G$, denoted by $\rho_{\alpha}(G)$, is the largest eigenvalue of $A_{\alpha}(G)$. 
Clearly, $A_{0}(G)=A(G)$ and $ 2A_{\frac{1}{2}}(G) =Q(G).$
Besides, Nikiforov \cite{Nikiforov} posed the $A_{\alpha}$-spectral extrema problem:
\begin{prob}\label{prob-0}
Given a graph $F$, what is the maximum $\alpha$-index of a graph $G$ of order $n$, with no subgraph isomorphic to $F$?
\end{prob}
At present, there are plentiful relevant results about Problem \ref{prob-0} when $F$ is a specific minor. When $F$ is a $K_r$-minor, Shi and Hong \cite{Shi} considered the case $r=4$ for $\alpha =0$,  Hong \cite{Hong} investigated the case $r=4$ for $\alpha =0$, Tait \cite{M.T} discussed the case $r\geq 3$ for sufficiently large $n$ and $\alpha =0$, while Chen, Liu and Zhang \cite{C.M} considered the case $r\geq 3$ for sufficiently large $n$ and $0 < \alpha < 1$. All the above studies indicate that Problem \ref{prob-0} has been completely solved when $F$ is a $K_r$-minor ($r\geq 3$) for sufficiently large $n$ and $0 \leq \alpha < 1$. Besides, there is a lot of discussion about Problem \ref{prob-0} when $F$ is a $K_{s,t}$-minor, see \cite{Nikiforov1, M.Q1, Nikiforov2, M.T, B.W, M.Q, M.F, C.M2,C.M, Y.T1, Y.T}.  Very recently, He, Li and Feng \cite{Xiao} showed that $K_s \vee \overline{K}_{n-s}$ (resp. $K_t \vee M_{n-t}$) is the unique extremal graph with maximum adjacency spectral radius
among all $F_{s}$-minor-free (resp. $Q_{t}$-minor-free) graphs of sufficiently large order $n$, where $s \geq 1$ (resp. $t\geq 1$).
Motivated by the above results of He, Li and Feng, 
in this paper we investigate Problem \ref{prob-0} when $F$ is an $F_{s}$-minor (resp. a $Q_{t}$-minor) for sufficiently large $n$ and $0<\alpha<1$, where $s \geq 1$ (resp. $t\geq 1$). The main results of this paper are as follows. 
\begin{thm}\label{thm::1.1}
Let $0<\alpha<1$ and $s \geq 1$. If $G$ is an $F_{s}$-minor-free graph of sufficiently large order $n$, then 
$$
\rho_{\alpha}(G) \leq \rho_{\alpha} (K_{s} \vee \overline{K}_{n-s}),
$$
with equality if and only if $G=K_{s} \vee \overline{K}_{n-s}$.
\end{thm}

\begin{thm}\label{thm::1.2}
Let $0<\alpha<1$ and $t \geq 1$. If $G$ is a $Q_t$-minor-free graph of sufficiently large order $n$, then 
$$
\rho_{\alpha}(G) \leq \rho_{\alpha} (K_t \vee M_{n-t} ),
$$
with equality if and only if $G=K_t \vee M_{n-t}$.
\end{thm}	

Taking $\alpha =\frac{1}{2}$ in Theorem \ref{thm::1.1} (resp. Theorem \ref{thm::1.2}), we obtain the unique extremal graph with maximum signless Laplacian radius among all $F_{s}$-minor-free (resp. $Q_{t}$-minor-free) graphs of sufficiently large order $n$ as follows.

\begin{cor}\label{cor::1.1}
Let $s \geq 1$ and $G$ be an $F_{s}$-minor-free graph of sufficiently large order $n$. Then 
$$
q(G) \leq q (K_{s} \vee \overline{K}_{n-s}),
$$
with equality if and only if $G=K_{s} \vee \overline{K}_{n-s}$.
\end{cor}

\begin{cor}\label{cor::1.2}
Let $t \geq 1$ and $G$ be a $Q_t$-minor-free graph of sufficiently large order $n$. Then 
$$
q(G) \leq q(K_t \vee M_{n-t}),
$$
with equality if and only if $G=K_t \vee M_{n-t}$.
\end{cor}	

\section{Preliminaries}
In this section, we will list some symbols and  useful lemmas. 
For $A, B \subseteq V(G)$, $e(A, B)$ (resp. $E(A, B)$) is denoted the number  (resp. the set) of the edges of $G$ with one end vertex in $A$ and the other in $B$, and particularly, $e(A, A)$ (resp. $E(A, A)$) is simplified by $e(A)$ (resp. $E(A)$). Let $G[A,B]$ be the \textit{complete bipartite subgraph} of $G$ with vertex partition $A$ and $B$.
For a vertex  $v \in V(G)$, we write $N_{G}(v)$ (resp. $d_{G}(v)$) for the set (resp. the number) of neighbors of $v$ in $G$, that is, $d_{G}(v)=|N_{G}(v)|$. Let $A $ be a subset of $V(G)$. We write $d_{A}(v)$ for the number of neighbors of  $v$  in $A$, that is,  $d_{A}(v)=|N_{G}(v) \cap A|$. The graph $G[A]$ is the \textit{induced subgraph} by $A$. $G[A]$ is called a \textit{clique} if it is a complete subgraph of $G$. 
For graph notations and concepts undefined here, readers are referred to \cite {Bondy}.

\begin{lem}(\cite{Nikiforov})\label{lem::2.1-}
Let $0\leq\alpha<1$. Then the $\alpha$-index of any proper subgraph of a connected graph is smaller than the $\alpha$-index of the original graph.
\end{lem}

\begin{lem} (\cite{C.M}) \label{lem::2.3}
Let $0<\alpha<1$, $n \geq k\geq1$ and $G=K_{k} \vee \overline{K}_{n-k}$. Then $\rho_{\alpha}(G) \geq \alpha(n-1)+(1-\alpha)(k-1)$. In particular, if $n \geq \frac{(2 k-1)^{2}}{2 \alpha^{2}}- \frac{8 k^{2}-2 k-1}{2 \alpha}+2 k(k+1)$, then $\rho_{\alpha}(G) \geq \alpha n+\frac{2 k-1-(2 k+1) \alpha}{2 \alpha}$.
\end{lem}

\begin{lem} (\cite{Cioabă}) \label{lem::2.6}
Let $N_1, \cdots, N_k$ be $k$ finite sets. Then
$$
\left|N_1 \cap N_2 \cap \cdots \cap N_k\right| \geq \sum_{i=1}^k\left|N_i\right|-(k-1)\left|\bigcup_{i=1}^k N_i\right|.
$$
\end{lem}

\begin{lem} (\cite{Mader}) \label{lem::2.5}
Let $G$ be an $n$-vertex graph. For every graph $H$, if $G$ is $H$-minor-free, then there exists a constant $C$ such that
$$
e(G) \leq C n.
$$
\end{lem}

\begin{lem} (\cite{Xiao}) \label{lem::2.1}
Let $s\geq1$ and $G$ be an $n$-vertex $F_{s}$-minor-free bipartite graph with vertex partition $A$ and $B$. If $|A|=a$ and $|B|=n-a$, then there exists a constant $C$ depending only on $s$ such that
$$
e(G) \leq C a+s n.
$$
\end{lem}

\begin{lem} (\cite{Xiao}) \label{lem::2.2}
Let $s\geq1$ and $G$ be an $n$-vertex $F_{s}$-minor-free graph. Suppose that $G$ contains a complete bipartite subgraph $K_{s,(1-\delta) n}=G[A, B]$, where $|A|=s$ and $|B|=(1-\delta)n$. Then the following two assertions are hold:\\
(i) if $(1-\delta) n \geq 2 s$, then $G[B]$ is $P_{2}$-free, and $d_{B}(v)\leq 1$ for any $v \in V(G) \backslash(A \cup B)$;\\
(ii) if $(1-2\delta)n \geq 2s+1$ and $G^{\prime}$ is a graph obtained from $G$ by adding edges to $A$ to make it a clique, then $G^{\prime}$ is also $F_{s}$-minor-free.
\end{lem}

\begin{lem} (\cite{Xiao}) \label{lem::2.7}
Let $t\geq1$ and $G$ be an $n$-vertex $Q_{t}$-minor-free bipartite graph with vertex partition $A$ and $B$. If $|A|=a$ and $|B|=n-a$, then there exists a constant  $C$ depending only on $t$ such that
$$
e(G) \leq C a+t n.
$$
\end{lem}

\begin{lem} (\cite{Xiao}) \label{lem::2.9}
Let $t\geq1$ and $G$ be an $n$-vertex $Q_t$-minor-free graph. Suppose that $G$ contains a complete bipartite subgraph $K_{t,(1-\delta) n}=G[A, B]$, where $|A|=t$ and $|B|=(1-\delta)n$. Then the following two assertions are hold:\\
(i) if $(1-\delta) n \geq 2 t+1$, then $G[B]$ is $P_3$-free, and $d_B(v)\leq 2$ for any $v \in V(G) \backslash(A \cup B)$;\\
(ii) if $(1-3\delta)n \geq 3t+1$ and $G^{\prime}$ is a graph obtained from $G$ by adding edges to $A$ to make it a clique, then $G^{\prime}$ is also $Q_t$-minor-free.
\end{lem}

\section{Proof of Theorem \ref{thm::1.1}}
\renewcommand\proofname{\bf Proof of Theorem \ref{thm::1.1}.}
\begin{proof}
Let $0<\alpha<1$, $s \geq 1$ and $G^{*}$ be an extremal graph with the maximum $\alpha$-index among all $F_{s}$-minor-free graphs of sufficiently large order $n$. 
Set for short $\rho_{\alpha}^{*}=\rho_{\alpha}(G^{*})$ and $\mathbf{x} = (x_{u})_{u\in V (G^{*})}$ with the maximum entry $1$ be a positive vector of $A_{\alpha}(G^{*})$ corresponding to $\rho_{\alpha}^{*}$. 
\renewcommand\proofname{\bf Proof}
\begin{clm}\label{clm::3.1-}
$G^{*}$ is connected.
\end{clm}
\begin{proof}
Suppose to the contrary that $G^{*}$ is disconnected. Let $G_1$ be a component with the largest $\alpha$-index in $G^{*}$. Then $\rho_{\alpha}^{*}=\rho_{\alpha}(G_1)$. Now, take an arbitrary vertex $u \in V(G^{*}) \setminus V(G_1)$. Let $G^{\prime}$ be a graph obtained from $G^{*}$ by removing all edges incident with $u$ and then connecting $u$ to a vertex  $v\in V(G_1)$. Clearly, $G^{\prime}$ is still $F_s$-minor-free and $\rho_{\alpha}(G^{\prime})=\rho_{\alpha}(G_1 + uv)$. Since $G_1$ is a proper subgraph of $G_1 + uv$, we have $\rho_{\alpha}(G_1 + uv)>\rho_{\alpha}(G_1)$ by Lemma \ref{lem::2.1-}. It follows that $\rho_{\alpha}(G^{\prime}) > \rho_{\alpha}^{*}$, a contradiction.
\end{proof}	
	
\begin{clm}\label{clm::3.1}
$\rho_{\alpha}^{*}\geq \max \left\lbrace  \alpha(n-1), \alpha n+\frac{2 s-1-(2 s+1) \alpha}{2 \alpha} \right\rbrace $.
\end{clm}
\begin{proof}
Note that $K_{s} \vee \overline{K}_{n-s}$ is $F_{s}$-minor-free. Then by Lemma \ref{lem::2.3} we obtain
\begin{align*}
\rho_{\alpha}^{*}\geq \rho_{\alpha}(K_{s} \vee \overline{K}_{n-s}) &\geq \max \left\lbrace  \alpha(n-1)+(1-\alpha)(s-1), \alpha n+\frac{2 s-1-(2 s+1) \alpha}{2 \alpha} \right\rbrace\\
&\geq \max \left\lbrace  \alpha(n-1), \alpha n+\frac{2 s-1-(2 s+1) \alpha}{2 \alpha} \right\rbrace.
\end{align*}
\end{proof}
	
Now, take an arbitrary vertex $u^{*} \in V (G^{*})$ with 
$$x_{u^{*}} =\max\{x_{u} : u \in V (G^{*})\}= 1.$$ 
Let 
$$L=\left\{u \in V(G^{*}): x_{u}>\epsilon\right\} \quad \text{and} \quad S=\left\{u \in V(G^{*}): x_{u} \leq \epsilon\right\},$$
where $\epsilon$ is a small constant which will be chosen later. Clearly, $V(G^{*})=L \cup S$.
\begin{clm}\label{clm::3.2}
There exists a constant $C_{1}$ such that $2 e(S) \leq 2 e(G^{*}) \leq C_{1} n.$
\end{clm}
\begin{proof}
It can be obtained directly by taking $C_{1}:=2 C$ in Lemma \ref{lem::2.5}.
\end{proof}
	
\begin{clm}\label{clm::3.3}
$e(L, S) \leq(s+\epsilon) n$, $2 e(L) \leq \epsilon n$ and $|L|\leq \epsilon n$.
\end{clm}
\begin{proof}
For any vertex $u \in L$, by $A_{\alpha}(G^{*})\mathbf{x}=\rho_{\alpha}^{*}\mathbf{x}$ we have
$$
\left(\rho_{\alpha}^{*}-\alpha d_{G^{*}}(u)\right) \epsilon<\left(\rho_{\alpha}^{*}-\alpha d_{G^{*}}(u)\right) x_{u}=(1-\alpha) \sum_{u v \in E(G^{*})} x_{v} \leq(1-\alpha) d_{G^{*}}(u),
$$
which yields that
$d_{G^{*}}(u)>\frac{\rho_{\alpha}^{*}\epsilon}{1-\alpha+\alpha \epsilon}.$
Hence, we obtain
\begin{equation}\label{equ::1.1}
2 e(G^{*})=\sum_{u\in V(G^{*})} d_{G^{*}}(u) \geq \sum_{u\in L} d_{G^{*}}(u) \geq \frac{|L| \rho_{\alpha}^{*} \epsilon}{1-\alpha+\alpha \epsilon}.
\end{equation} 
Since $n$ is sufficiently large, there is a constant $C_2$ such that
$$
C_2 \geq \frac{2 \alpha C_{1}}{2 \alpha^{2}+\frac{2 s-1-(2s+1) \alpha}{n}}.
$$
By (\ref{equ::1.1}) and Claims \ref{clm::3.1}-\ref{clm::3.2}, we get that
\begin{align}
|L| &\leq \frac{2 e(G^{*})}{\rho_{\alpha}^{*}} \cdot \frac{1-\alpha+\alpha \epsilon}{\epsilon} \leq \frac{C_{1} n}{\alpha n+\frac{2 s-1-(2 s+1) \alpha}{2 \alpha}} \cdot \frac{1-\alpha+\alpha \epsilon}{\epsilon} \nonumber \\
& =\frac{2 \alpha C_{1}}{2 \alpha^{2}+\frac{2 s-1-(2 s+1) \alpha}{n}} \cdot \frac{1-\alpha+\alpha \epsilon}{\epsilon} \leq \frac{C_2(1-\alpha+\alpha \epsilon)}{\epsilon}.     \label{equ::1.4}
\end{align}
This indicates that $|L|\leq \epsilon n$ as long as $n$ is sufficiently large such that $n \geq \frac{C_2(1-\alpha+\alpha \epsilon)}{\epsilon^{2}}$. By Lemma \ref{lem::2.1} and (\ref{equ::1.4}), there is a constant $C_3$ depending only on $s$ such that
$$
e(L, S) \leq C_3|L|+s n \leq \frac{C_3C_2(1-\alpha+\alpha \epsilon)}{\epsilon}+s n \leq(s+\epsilon) n,
$$
where the last inequality holds as long as $n$ is sufficiently large such that $n \geq \frac{C_3 C_2(1-\alpha+\alpha \epsilon)}{\epsilon^{2}}$.
Furthermore, by Lemma \ref{lem::2.5} and (\ref{equ::1.4}), we obtain
$$
2e(L) \leq C_{1}|L| \leq \frac{C_{1} C_2(1-\alpha+\alpha \epsilon)}{\epsilon} \leq \epsilon n,
$$
where the last inequality holds as long as $n$ is sufficiently large such that $n \geq \frac{C_{1} C_2(1-\alpha+\alpha \epsilon)}{\epsilon^{2}}$.
\end{proof}
	
\begin{clm}\label{clm::3.4}
For any vertex $u \in L$, there exists a constant $C_{4}$ ($C_{4} \geq1$ is possible) such that
$$
d_{G^{*}}(u) \geq\left( 1-C_{4}(1-x_{u}+\epsilon)\right) n.
$$
\end{clm}
\begin{proof}
By $A_{\alpha}(G^{*})\mathbf{x}=\rho_{\alpha}^{*}\mathbf{x}$ and Claims \ref{clm::3.2}-\ref{clm::3.3}, we can see that
$$
\begin{aligned}
\rho_{\alpha}^{*} \sum_{v \in V(G^{*})} x_{v} 
= & \sum_{v \in V(G^{*})} \rho_{\alpha}^{*} x_{v}=\sum_{v \in V(G^{*})}\left(\alpha d_{G^{*}}(v) x_{v}+(1-\alpha) \sum_{v z \in E(G^{*})} x_{z}\right) \\	
= & \alpha \sum_{v \in V(G^{*})} d_{G^{*}}(v) x_{v}+(1-\alpha) \sum_{v z \in E(G^{*})}\left(x_{v}+x_{z}\right) \\
= & \alpha\left(\sum_{v \in L} d_{G^{*}}(v) x_{v}+\sum_{v \in S} d_{G^{*}}(v) x_{v}\right)+(1-\alpha)\left(\sum_{v z \in E(L)}\left(x_{v}+x_{z}\right)+\right. \\
& \left.\sum_{v z \in E(L, S)}\left(x_{v}+x_{z}\right)+\sum_{v z \in E(S)}\left(x_{v}+x_{z}\right)\right) \\
\leq & \alpha(2 e(L)+e(L, S))+\alpha \epsilon(2 e(S)+e(L, S))+(1-\alpha)(2 e(L)+ \\
& (1+\epsilon) e(L, S)+2 \epsilon e(S)) \\
= & 2 e(L)+2 \epsilon e(S)+(1+\epsilon) e(L, S)\\
\leq& \epsilon n+\epsilon C_{1} n+(1+\epsilon)(s+\epsilon) n\\
=&\left(\left(1+C_{1}\right) \epsilon+(1+\epsilon)(s+\epsilon)\right) n,
\end{aligned}
$$
which implies that
\begin{equation*}
\begin{aligned}
\sum_{v \in V(G^{*})} x_{v}  \leq \frac{\left(\left(1+C_{1}\right) \epsilon+(1+\epsilon)(s+\epsilon)\right) n}{\rho_{\alpha}^{*}}.
\end{aligned}
\end{equation*}
For any vertex $u \in L$,  it is easy to see that
\begin{equation*}
\left(\rho_{\alpha}^{*}-\alpha d_{G^{*}}(u)\right) x_{u}=(1-\alpha) \sum_{u v \in E(G^{*})} x_{v} \leq(1-\alpha) \sum_{v \in V(G^{*})} x_{v}.
\end{equation*}
Combining with Claim \ref{clm::3.1}, there exists a constant $C_4$ ($C_{4} \geq1$ is possible) such that
$$
\begin{aligned}
d_{G^{*}}(u) & \geq \frac{\rho_{\alpha}^{*}}{\alpha}-\frac{(1-\alpha) \sum_{v \in V(G^{*})} x_{v}}{\alpha x_{u}} \\
& \geq \frac{\rho_{\alpha}^{*}}{\alpha}-\frac{(1-\alpha)\left(\left(1+C_{1}\right) \epsilon+(1+\epsilon)(s+\epsilon)\right) n}{\rho_{\alpha}^{*} \alpha x_{u}} \\
&\geq n-1-\frac{(1-\alpha)\left(\left(1+C_{1}\right) \epsilon+(1+\epsilon)(s+\epsilon)\right)n}{\alpha^{2}\left(n-1\right) x_{u}}\\
&\geq n-\frac{\alpha^{2} x_{u}+(1-\alpha)\left(\left(1+C_{1}\right) \epsilon+(1+\epsilon)(s+\epsilon)\right)}{\alpha^{2} x_{u}}\cdot\frac{n}{n-1}\\
&\geq \left(1-C_{4}\left(1-x_{u}+\epsilon\right)\right)n,
\end{aligned}
$$
where the last inequality holds as long as $n$ is sufficiently large such that $$
n \geq \frac{\alpha^{2} x_{u}+(1-\alpha)\left(\left(1+C_{1}\right) \epsilon+(1+\epsilon)(s+\epsilon)\right)}{\alpha^{2} x_{u}C_{4}\left(1-x_{u}+\epsilon\right)}+1.
$$
\end{proof}	
	
\begin{clm}\label{clm::3.4-}
$\frac{\rho_{\alpha}^{*}\left(\rho_{\alpha}^{*}+1-\alpha n \right)}{1-\alpha} \geq(n-s)s.$
\end{clm}
\begin{proof}
Denote by $G:= K_{s} \vee \overline{K}_{n-s}$. Let $\rho_{\alpha}= \rho_{\alpha}\left(G\right)$ and
$\mathbf{x^{\prime}}$ be a positive eigenvector of $A_{\alpha}(G)$ corresponding to $\rho_{\alpha}$. By symmetry and the Perron-Frobenius theorem, all vertices of subgraphs $K_{s}$ or $\overline{K}_{n-s}$ in $G:= K_{s} \vee \overline{K}_{n-s}$ have the same eigenvector components respectively, which are denoted by $x_{1}^{\prime}$ and $x_{2}^{\prime}$, respectively.  By $A_{\alpha}(G)\mathbf{x^{\prime}}=\rho_{\alpha}\mathbf{x^{\prime}}$, we can see that
$$
\begin{aligned}
(\rho_{\alpha}-\alpha(n-1)-(1-\alpha)(s-1)) x_{1}^{\prime} &=(1-\alpha)(n-s) x_{2}^{\prime}, \\
(\rho_{\alpha}-\alpha s)x_{2}^{\prime} &=(1-\alpha)s x_{1}^{\prime}.
\end{aligned}
$$
Then $\rho_{\alpha}$ is the largest root of $g(x) = 0$, where
$$
g(x) =x^2 - (\alpha n + s - 1)x + (2 \alpha n - \alpha s- \alpha - n + s)s.
$$
Moreover, let
$$
f(x) =x^2 - (\alpha n - 1)x -(1-\alpha)(n-s)s.
$$
Clearly,
$$g(x)=f(x)-s(x-\alpha(n-1)).$$
By Lemma \ref{lem::2.3} we get that
$$\rho_{\alpha} \geq \alpha(n-1)+(1-\alpha)(s-1) \geq \alpha(n-1) \geq \alpha n-1.$$ 
Thus we obtain
$\rho_{\alpha} \geq\frac{\alpha n - 1}{2}$ and
$$
f(\rho_{\alpha})=s(\rho_{\alpha}-\alpha(n-1))\geq0,
$$
which imply that $\rho_{\alpha}$ is no less than the largest root of $f(x) = 0$. Furthermore, since $G$ is $F_{s}$-minor-free, we have $\rho_{\alpha}^* \geq \rho_{\alpha}$. Therefore, $\rho_{\alpha}^*$ is also no less than the largest root of $f(x) = 0$, that is,
$$f(\rho_{\alpha}^*) =\rho_{\alpha}^*\left( \rho_{\alpha}^*+1- \alpha n\right) -(1-\alpha)(n-s)s \geq 0,$$ 
the result follows.
\end{proof}	
	
\begin{clm}\label{clm::3.5}
There exists a constant $C_5$ and $s$ distinct vertices $v_1, v_2, \dots, v_s \in L$ such that $x_{v_i} \geq 1-C_5\epsilon$ and $d_{G^*}(v_i)\geq  (1-C_5\epsilon)n$ for each $i\in \{1, 2, \dots, s\}$. 	
\end{clm}
\begin{proof}	
We will prove this result by using induction. Since $x_{u^*}=1$, we may first set $v_1=u^*$, by Claim \ref{clm::3.4} we have $d_{G^*}(v_1) \geq (1-c_1 \epsilon)n$, where $c_1=C_4$, as desired. Now, suppose that there is a constant $c_d$ and $d$ ($1 \leq d\leq s-1$) distinct vertices $v_1, v_2, \dots, v_d \in L$ such that $x_{v_i} \geq 1-c_d\epsilon$ and $d_{G^*}(v_i)\geq  (1-c_d\epsilon)n$ for each $i\in \{1, 2, \dots, d\}$. We next show that there exists a constant $c_{d+1} > c_d$ and a vertex $v_{d+1} \in L \backslash \{v_1, v_2, \dots, v_d\}$ such that $x_{v_{d+1}} \geq 1-c_{d+1}\epsilon$ and $d_{G^{*}}(v_{d+1}) \geq\left(1-c_{d+1}\epsilon\right)n$.
		
Let $D=\{v_1, v_2, \dots, v_d\} \subseteq L$.
By $A_{\alpha}(G^{*})\mathbf{x}=\rho_{\alpha}^{*}\mathbf{x}$ and $x_{u^{*}} =1$, we have
\begin{align*}
&\rho_{\alpha}^{*}\left(\rho_{\alpha}^{*}-\alpha d_{G^{*}}(u^{*})\right)\\
=&\rho_{\alpha}^{*}\left( \left(\rho_{\alpha}^{*}-\alpha d_{G^{*}}(u^{*})\right) x_{u^{*}}\right) =\rho_{\alpha}^{*}\left( (1-\alpha) \sum_{v u^{*} \in E(G^{*})} x_{v}\right)=(1-\alpha) \sum_{v u^{*} \in E(G^{*})} \rho_{\alpha}^{*}x_{v}\\
=&(1-\alpha) \sum_{v u^{*} \in E(G^{*})} \left( \alpha d_{G^{*}}(v) x_v +(1-\alpha) \sum_{uv \in E(G^{*})} x_u \right)\\
=&(1-\alpha) \sum_{v u^{*} \in E(G^{*})} \alpha d_{G^{*}}(v) x_v+(1-\alpha)^{2} \sum_{v u^{*} \in E(G^{*})}\sum_{uv \in E(G^{*})} x_u\\
\leq & (1-\alpha)\left(\sum_{v \in V(G^{*})} \alpha d_{G^{*}}(v) x_{v}-\alpha d_{G^{*}}(u^{*})\right)+(1-\alpha)^{2} \left( \sum_{u v \in E(G^{*})}\left(x_{u}+x_{v}\right)-\sum_{v u^{*} \in E(G^{*})} x_{v}\right)  \\
= & \alpha(1-\alpha) \sum_{u v \in E(G^{*})}\left(x_{u}+x_{v}\right)-\alpha(1-\alpha) d_{G^{*}}(u^{*})+ (1-\alpha)^{2} \sum_{u v \in E(G^{*})}\left(x_{u}+x_{v}\right)-\\
& (1-\alpha)\left(\rho_{\alpha}^{*}-\alpha d_{G^{*}}(u^{*})\right) \\
= & (1-\alpha) \sum_{u v \in E(G^{*})}\left(x_{u}+x_{v}\right)-(1-\alpha) \rho_{\alpha}^{*}.
\end{align*}
Note that $d_{G^{*}}(u^{*})\leq n-1$. Combining with Claim \ref{clm::3.4-} we obtain $$
\sum_{u v \in E(G^{*})}\left(x_{u}+x_{v}\right) \geq \frac{\rho_{\alpha}^{*}\left(\rho_{\alpha}^{*}+1-\alpha-\alpha d_{G^{*}}(u^{*})\right)}{1-\alpha} \geq\frac{\rho_{\alpha}^{*}\left(\rho_{\alpha}^{*}+1-\alpha n\right)}{1-\alpha} \geq (n-s)s.
$$
On the other hand, recall that $D \subseteq L$, $V(G^{*})=L \cup S$, $x_u > \epsilon$ for any $u \in L$ and $x_u \leq \epsilon$ for any $u \in S$. Then by Claims \ref{clm::3.2}-\ref{clm::3.3}, we get that
$$
\begin{aligned}
\sum_{u v \in E(G^{*})}\left(x_{u}+x_{v}\right) 
= & \sum_{u v \in E(L, S)}\left(x_{u}+x_{v}\right)+\sum_{u v \in E(S)}\left(x_{u}+x_{v}\right)+\sum_{u v \in E(L)}\left(x_{u}+x_{v}\right) \\
\leq & \sum_{u v \in E(L, S)}\left(x_{u}+x_{v}\right)+2 \epsilon e(S)+2 e(L) \\
\leq & \epsilon e(L, S)+\sum_{\substack{u v \in E(L \backslash D, S) \\
u \in L \backslash D}} x_{u}+\sum_{\substack{u v \in E(D, S) \\
u \in D}} x_{u}+2 \epsilon e(S)+2 e(L)\\
\leq & \epsilon(s+\epsilon) n+\sum_{\substack{u v \in E(L \backslash D, S) \\
u \in L \backslash D}} x_{u}+dn+\epsilon C_{1} n+\epsilon n.
\end{aligned}
$$
Therefore, we have
$$
\begin{aligned}
\sum_{\substack{u v \in E(L \backslash D, S) \\ u \in L \backslash D}} x_{u}
\geq & (n-s)s-\epsilon(s+\epsilon) n-dn- \epsilon C_{1} n-\epsilon n
=  (s-d-\epsilon(C_{1}+\epsilon+s+1)) n-s^{2}\\
\geq & (s-d-\epsilon(C_{1}+\epsilon+s+2)) n,
\end{aligned}
$$
where the last inequality holds as long as $n$ is sufficiently large such that $n \geq \frac{s^2}{\epsilon}$.
Furthermore, by Claim \ref{clm::3.3} and
$$
e(D, S)+e(D, L \backslash D)+2 e(D)=\sum_{v \in D} d_{G^{*}}(v) \geq d\left(1-c_d \epsilon\right)n,
$$
we find that
$$
\begin{aligned}
e(D, S) & \geq d\left(1-c_d \epsilon\right) n-e(D, L \backslash D)-2 e(D)  \geq d\left(1-c_d \epsilon\right) n-d(|L|-d)-d(d-1) \\
& \geq d\left(1-c_d \epsilon\right) n-d(\epsilon n-d)-d(d-1)  =d\left(1-c_d \epsilon-\epsilon\right) n+d >d\left(1-c_d \epsilon-\epsilon\right) n,
\end{aligned}
$$
and so
$$
\begin{aligned}
e(L \backslash D, S)=e(L, S)-e(D, S)<(s+\epsilon) n-d\left(1-c_d \epsilon-\epsilon\right) n=\left( s+\epsilon-d\left(1-c_d \epsilon-\epsilon\right)\right)  n.
\end{aligned}
$$
Let 
$$
\varphi(x)=\frac{s-x-\epsilon\left(C_{1}+\epsilon+s+2 \right)}{ s+\epsilon-x\left(1-c_d \epsilon-\epsilon\right)}.
$$
Then
$$ 
\begin{aligned}
\varphi^{\prime}(x)=\dfrac{-(s+\epsilon)+(1-c_d \epsilon-\epsilon)(s-\epsilon\left(C_{1}+\epsilon+s+2 \right))}{ (s+\epsilon-x\left(1-c_d \epsilon-\epsilon\right))^{2}}<0,
\end{aligned}
$$
which implies that $\varphi(x)$ is decreasing with respect to $1 \leq x \leq s -1$.
Thus, by averaging, there exists a vertex $v_{d+1} \in L \backslash D$ such that
$$
\begin{aligned}
x_{v_{d+1}} &\geq \frac{\sum_{\substack{u v \in E(L \backslash D, S) \\ u \in L \backslash D}} x_{u}}{e(L \backslash D, S)}
\geq \varphi(d)\\ &\geq \varphi(s-1)=\frac{1-\epsilon\left(C_{1}+\epsilon+s+2 \right)}{ 1+\epsilon+(s-1)\left(c_d \epsilon+\epsilon\right)} = 1-\frac{\epsilon\left(C_{1}+\epsilon+s+3\right)+(s-1)\left(c_d \epsilon+\epsilon\right)}{1+\epsilon+(s-1)\left(c_d \epsilon+\epsilon\right)}\\
&\geq 1-\frac{\left( \epsilon+c_d \epsilon\right) \left(C_{1}+s+3\right)+(s-1)\left(c_d \epsilon+\epsilon\right)}{1+\epsilon+(s-1)\left(c_d \epsilon+\epsilon\right)}\geq 1-\left(C_1+2s+2 \right) \left(c_d \epsilon+\epsilon\right).
\end{aligned}
$$
Then by Claim \ref{clm::3.4} we have
$$
d_{G^{*}}\left(v_{d+1}\right) \geq  \left(1-C_{4}\left( \left(C_1+2s+2 \right) \left(c_d \epsilon+\epsilon\right)+\epsilon\right) \right)n \geq \left(1-C_{4} \left(C_1+2s+3 \right) \left(c_d +1\right)\epsilon \right)n.
$$
Denote by $c_{d+1} :=C_{4}\left(C_1+2s+3 \right) \left(c_d +1\right)$. Then
$d_{G^{*}}(v_{d+1}) \geq (1-c_{d+1}\epsilon)n$. Setting $C_4 \geq1$ in Claim \ref{clm::3.4}, we obtain $c_{d+1} > c_d$. 
It follows that $x_{v_{i}} \geq 1-c_{d+1}\epsilon$ and $d_{G^{*}}(v_{i}) \geq\left(1-c_{d+1}\epsilon\right)n$ for each $i\in \{1, 2, \dots, d+1\}$, as desired.
\end{proof}	
	
Let the $s$ distinct vertices $v_1, v_2, \dots, v_s \in L$ be defined in Claim \ref{clm::3.5}.
Denote by $A=\left\{v_{1}, v_{2}, \ldots, v_{s}\right\}$, $B= \bigcap_{i=1}^{s} N_{G^*}\left(v_{i}\right)$ and $R=V(G^*) \backslash(A \cup B)$. Then by Lemma \ref{lem::2.6} and Claim \ref{clm::3.5}, we obtain
$$
\begin{aligned}
|B| \geq \sum_{i=1}^{s}d_{G^*}(v_{i})-(s-1)\left| \bigcup_{i=1}^{s} N_{G^*}(v_{i})\right|  \geq s(1-C_5\epsilon) n-(s-1) n=(1-C_5\epsilon s ) n,
\end{aligned}
$$
and so
\begin{equation}\label{equ::1.5}
|R|=n-|A|-|B| \leq C_5\epsilon s n.
\end{equation}
\begin{clm}\label{clm::3.6}
$G^*[A]$ is a clique. Moreover,  $B$ is an independent set.
\end{clm}
\begin{proof}
Obviously, $G^*$ contains a complete bipartite subgraph $G^*[A, B]$ with $|A|=s$ and $|B|=(1-\delta) n$, where $\delta \geq C_5\epsilon s$. Furthermore,  $(1-2 \delta) n \geq 2 s+1$ for sufficiently large $n$. Suppose that $G^*[A]$ is not a clique. Let $G^{\prime}$ be the graph obtained from $G^*$ by adding edges to $A$ to make it a clique. Then $G^{\prime}$ is also $F_{s}$-minor-free by Lemma \ref{lem::2.2}(ii). Since $G^*$ is a proper subgraph of $G^{\prime}$, by Lemma \ref{lem::2.1-} we have $\rho_{\alpha}(G^{\prime}) > \rho_{\alpha}^{*}$, a contradiction. Hence, $A$ induces a clique in $G^*$. Moreover, since $(1-\delta) n\geq(1-2 \delta) n \geq 2 s+1> 2s$, we get that $G^*[B]$ is $P_{2}$-free by Lemma \ref{lem::2.2}(i), which implies that $B$ is an independent set.
\end{proof}
	
\begin{clm}\label{clm::3.8}
$x_{v} \leq \frac{\sqrt{2\alpha}}{2(C_{1}+1)}$ for any $v \in V(G^*) \backslash A$.
\end{clm}
\begin{proof}
For any vertex $v \in B$, since $B$ is an independent set from Claim \ref{clm::3.6}, we have $d_{B}(v)=0$, and so
\begin{equation}\label{equ::1.8}
|N_{G^*}(v) \cap(A \cup B)|=d_{A}(v)+d_{B}(v)=s.
\end{equation}
In addition, for any vertex $v\in R$, we see that $d_{A}(v)\leq s-1$ by the definition of $R$, and $d_{B}(v)\leq 1$ by Lemma \ref{lem::2.2}(i), thus
\begin{equation}\label{equ::1.6}
\left|N_{G^*}(v) \cap(A \cup B)\right|=d_{A}(v)+d_{B}(v) \leq s-1+1=s.
\end{equation}
Since $G^*[R]$ is $F_{s}$-minor-free, we have $2 e(R) \leq C_{1}|R|$  by Lemma \ref{lem::2.5}.
Furthermore, by $A_{\alpha}(G^{*})\mathbf{x}=\rho_{\alpha}^{*}\mathbf{x}$ and Claim \ref{clm::3.1}, we can see that
$$
\begin{aligned}
\alpha(n-1) \sum_{u \in R} x_{u} & \leq \rho_{\alpha}^* \sum_{u \in R} x_{u}=\sum_{u \in R}\left(\alpha d_{G^*}(u) x_{u}+(1-\alpha) \sum_{u v \in E(G^*)} x_{v}\right) \\
& \leq \sum_{u \in R}(\alpha d_{G^*}(u)+(1-\alpha) d_{G^*}(u))=\sum_{u\in R} d_{G^*}(u) \\
& \leq 2 e(R)+e(R, A\cup B)\leq C_{1}|R|+s|R|=(C_1+s)|R|,
\end{aligned}
$$
which implies that
\begin{equation}\label{equ::1.7}
\sum_{u \in R} x_{u} \leq \frac{(C_1+s)|R|}{\alpha(n-1)}.
\end{equation}
Now, take an arbitrary vertex $v \in V(G^*) \backslash A=B \cup R$. It follows form 
(\ref{equ::1.8})-(\ref{equ::1.6}) that
\begin{equation}\label{equ::1.11}
\left|N_{G^*}(v) \cap(A \cup B)\right| \leq s,
\end{equation}
and so
\begin{equation}\label{equ::1.9}
d_{G^*}(v) \leq \left|N_{G^*}(v) \cap(A \cup B)\right|+|R|\leq s+|R|.
\end{equation}
Moreover, by $A_{\alpha}(G^{*})\mathbf{x}=\rho_{\alpha}^{*}\mathbf{x}$ and (\ref{equ::1.7})-(\ref{equ::1.11}), we obtain
\begin{align*}
\left(\rho_{\alpha}^*-\alpha d_{_{G^*}}(v)\right) x_{v}&=(1-\alpha) \sum_{u v \in E(G^*)} x_{u} =(1-\alpha)\left(\sum_{\substack{u \in N_{G^*}(v) \\ u \in A \cup B}} x_{u}+\sum_{\substack{u \in N_{G^*}(v) \\ u \in R}} x_{u}\right) \\
&\leq(1-\alpha)\left(s+\sum_{u \in R} x_{u}\right)\leq(1-\alpha)\left(s+\frac{(C_1+s)|R|}{\alpha(n-1)}\right).
\end{align*}
Combining with Claim \ref{clm::3.1}, (\ref{equ::1.5}) and (\ref{equ::1.9}), we get that
\begin{align*}
x_{v}  &\leq \frac{(1-\alpha)\left(s+\frac{(C_1+s)|R|}{\alpha(n-1)}\right)}{\rho_{\alpha}^*-\alpha d_{G^*}(v)} \leq \frac{(1-\alpha)\left(s+\frac{(C_1+s)|R|}{\alpha(n-1)}\right)}{\alpha(n-1)-\alpha(s+|R|)}\leq \frac{(1-\alpha)\left(s+\frac{(C_1+s)C_5\epsilon s n}{\alpha(n-1)}\right)}{\alpha(n-1)-\alpha(s+C_5\epsilon s n)}\\
&\leq \frac{(1-\alpha)\cdot\frac{s\alpha+ \left(C_{1}+s\right) C_5\epsilon s}{\alpha}\cdot \frac{n}{n-1}}{\alpha((1-C_5\epsilon s) n-s-1)}\leq \frac{(1-\alpha)\cdot\frac{s\alpha+ \left(C_{1}+s\right) C_5\epsilon s}{\alpha}\cdot \frac{n}{n-1}}{\alpha(1-C_5\epsilon s-\epsilon) n}
 \leq \frac{\sqrt{2\alpha}}{2(C_{1}+1)},
\end{align*}
where the  second last inequality holds as long as $n$ is sufficiently large such that $n \geq \frac{s+1}{\epsilon}$, and the last inequality holds as long as $n$ is sufficiently large such that 
$$
n \geq \frac{2(C_{1}+1)(1-\alpha)(s\alpha+ \left(C_{1}+s\right) C_5\epsilon s)}{\sqrt{2\alpha}\alpha^2(1-C_5\epsilon s-\epsilon)}+1.
$$
\end{proof}
\begin{clm}\label{clm::3.9}
$R$ is empty.
\end{clm}
\begin{proof}
Suppose to the contrary that $R$ is not empty. Since $G^*[R]$ is  $F_{s}$-minor-free, by Lemma \ref{lem::2.5} we have
$2e(R) \leq C_{1}|R|$, which implies that the average degree of $G^*[R]$ is at most $C_1$. Hence, there exists a vertex $v_{1} \in R$ such that $d_{R}(v_{1})\leq C_1$. By reusing Lemma \ref{lem::2.5}, we can order the vertices $v_2, \ldots, v_{|R|} \in R\setminus\{v_1\}$ such that $v_{i} \in R\setminus \{v_1, \ldots, v_{i-1}\}$ with $d_{R\setminus \{v_1, \ldots, v_{i-1}\}}(v_{i})\leq C_1$ for each $i=2, \ldots, |R| $. Clearly, $R=\{v_1, v_2, \ldots, v_{|R|} \}$, $v_1$ has at most $C_1$ neighbors in $R$, and each $v_i$ has at most $C_1$ neighbors in $R\setminus \{v_1, \ldots, v_{i-1}\}$ for $i \in \{2, \ldots, |R|\}$. By the definition of $R$, each $v_i$ ($i \in \{1, 2, \ldots, |R|\}$) has at least one non-neighbor in $A$. Besides, it follows from Lemma \ref{lem::2.2}(i) that each vertex $v_i$ ($i \in \{1, 2, \ldots, |R|\}$) has at most one neighbor in $B$. Let
\begin{align*}
G^{\prime}=& G^*-\{v_{i}u \in E(G^*) \mid  i\in \{1, 2, \dots, |R|\}, u \in R\}-\\
&\{ v_{i}w \in E(G^*)\mid  i\in \{1, 2, \dots, |R|\}, w \in B\}
 +\{v_{i} z \notin E(G^*)\mid  i\in \{1, 2, \dots, |R|\}, z \in A\}.
\end{align*}
Clearly, $G^{\prime}=K_{s} \vee \overline{K}_{n-s}$, thus $G^{\prime}$ is also $F_s$-minor-free. Furthermore, by Claims \ref{clm::3.5} and \ref{clm::3.8} we obtain
\begin{align*}
\rho_\alpha\left(G^{\prime}\right)-\rho_\alpha^* \geq& \frac{\mathbf{x}^{T} A_\alpha\left(G^{\prime}\right) \mathbf{x}}{\mathbf{x}^{T} \mathbf{x}}-\frac{\mathbf{x}^{T} A_\alpha \left(G^{*}\right)\mathbf{x}}{\mathbf{x}^{T} \mathbf{x}} \\
=&\frac{1}{\mathbf{x}^{T} \mathbf{x}} \sum_{i=1}^{|R|} \left(\sum_{\substack{v_i z \notin E(G^*) \\
z \in A}} ( \alpha x_{v_i}^2+2(1-\alpha) x_{v_i} x_{z}+\alpha x_{z}^2) -\sum_{\substack{u v_i\in E(G^*) \\
u\in R}} ( \alpha x_{u}^2+ \right. \\
&\left.2(1-\alpha) x_{u} x_{v_i}+\alpha x_{v_i}^2)-\sum_{\substack{v_i w \in E(G^*) \\
w \in B}} ( \alpha x_{v_i}^2+2(1-\alpha) x_{v_i} x_{w}+\alpha x_{w}^2)\right) \\
\geq& \frac{|R|}{\mathbf{x}^{T} \mathbf{x}}\left(\alpha(1-C_5\epsilon)^{2} -\frac{ (C_1+1)(\alpha+2(1-\alpha)+\alpha )(\sqrt{2\alpha})^2}{(2(C_1+1))^{2}}\right) \\
=&\frac{|R|\alpha}{\mathbf{x}^{T} \mathbf{x}}\left((1-C_5\epsilon)^{2} -\frac{1}{C_1+1}\right) >0,
\end{align*}
where the last inequality holds if we choose $\epsilon$ small enough such that $(1-C_5\epsilon)^{2} > \frac{1}{C_1+1}$. Hence, $\rho_\alpha\left(G^{\prime}\right)>\rho_\alpha^*$, a contradiction. 
\end{proof}
By Claims \ref{clm::3.6} and \ref{clm::3.9}, we obtain $G^* = K_{s} \vee \overline{K}_{n-s}$, as desired. This completes the proof of Theorem \ref{thm::1.1}.
\end{proof}

\section{Proof of Theorem \ref{thm::1.2}}
\renewcommand\proofname{\bf Proof of Theorem \ref{thm::1.2}.}
\begin{proof}
Let $0<\alpha<1$, $t\geq 1$ and $G^{\ast}$ be an extremal graph with maximum $\alpha$-index among all $Q_{t}$-minor-free graphs of sufficiently large order $n$. Set for short $\rho_{\alpha}^{\ast}=\rho_{\alpha}(G^{\ast})$ and $\mathbf{x} = (x_{u})_{u\in V (G^{\ast})}$ with the maximum entry $1$ be a positive vector of $A_{\alpha}(G^{\ast})$ corresponding to $\rho_{\alpha}^{\ast}$. Since $K_{t} \vee \overline{K}_{n-t}$ is $Q_{t}$-minor-free, by Lemma \ref{lem::2.3} we obtain  
\begin{align}\label{equ::2.1}
\rho_{\alpha}^{*}\geq \rho_{\alpha}(K_{t} \vee \overline{K}_{n-t}) \geq \max \left\lbrace  \alpha(n-1), \alpha n+\frac{2 t-1-(2 t+1) \alpha}{2 \alpha} \right\rbrace.
\end{align} 
Similar to the proof of Claim \ref{clm::3.1-}, we get that $G^{\ast}$ is connected. 
Now, take an arbitrary vertex $u^{\ast} \in V (G^{\ast})$ with 
$$x_{u^{\ast}} =\max\{x_{u} : u \in V (G^{\ast})\}= 1.$$ 
Let 
$$L=\left\{u \in V(G^{\ast}): x_{u}>\epsilon\right\} \quad \text{and} \quad S=\left\{u \in V(G^{\ast}): x_{u} \leq \epsilon\right\},$$
where $\epsilon$ is a small constant which will be chosen later. Clearly, $V(G^{\ast})=L \cup S$. 
By Lemma \ref{lem::2.5}, 
there is a constant $C_{1}$ such that 
$$2 e(S) \leq 2 e(G^{*}) \leq C_{1} n.$$
Furthermore, similar to the proofs of Claims \ref{clm::3.3}-\ref{clm::3.5}, if we choose $\epsilon$ small enough, then there exists a constant $C^{\prime}$ and a set $A=\{v_{1}, v_{2}, \dots, v_{t}\}\subseteq L$ such that
\begin{equation}\label{equ::2.2}
x_{v_i} \geq 1-C^{\prime}\epsilon,
\end{equation}
and $d_{G^*}(v_i)\geq (1-C^{\prime}\epsilon)n$ for each $i=1, 2, \dots, t$. Let
$B= \bigcap_{i=1}^{t} N_{G^*}\left(v_{i}\right)$ and $R=V(G^*) \backslash(A \cup B)$. Then by Lemma \ref{lem::2.6} we obtain
$$
\begin{aligned}
|B| \geq \sum_{i=1}^{t}d_{G^*}(v_{i})-(t-1)\left| \bigcup_{i=1}^{t} N_{G^*}(v_{i})\right|  \geq t(1-C^{\prime}\epsilon) n-(t-1) n=(1-C^{\prime}\epsilon t ) n,
\end{aligned}
$$
and so
\begin{equation}\label{equ::2.3}
|R|=n-|A|-|B| \leq C^{\prime}\epsilon t n.
\end{equation}
\renewcommand\proofname{\bf Proof}
\begin{clm}\label{clm::4.2}
$G^*[A]$ is a clique. Moreover, $G^*[B]$ consists of some independent edges and isolated vertices.
\end{clm}
\begin{proof}
Obviously, $G^*$ contains a complete bipartite subgraph $G^*[A, B]$ with $|A|=t$ and  $|B|=(1-\delta) n$, where $\delta \geq C^{\prime}\epsilon t$. Furthermore,  $(1-3 \delta) n \geq 3 t+1$ for sufficiently large $n$. Suppose  that $G^*[A]$ is not a clique. Let $G^{\prime}$ be the graph obtained from $G^*$ by adding edges to $A$ to make it a clique. Then $G^{\prime}$ is also $Q_{t}$-minor-free by Lemma \ref{lem::2.9}(ii). Since $G^*$ is a proper subgraph of $G^{\prime}$, we have $\rho_{\alpha}(G^{\prime}) > \rho_{\alpha}^{*}$ by Lemma \ref{lem::2.1-}, a contradiction. Hence, $A$ induces a clique in $G^*$. Moreover, since $(1-\delta) n\geq(1-3 \delta) n \geq 3 t+1> 2t+1$, we see that $G^*[B]$ is $P_{3}$-free by Lemma \ref{lem::2.9}(i), which implies that $G^*[B]$ consists of some independent edges and isolated vertices.
\end{proof}
	
\begin{clm}\label{clm::4.3}
$x_{v} \leq \frac{\sqrt{2\alpha}}{2(C_{1}+2)}$ for any $v \in V(G^*) \backslash A$.
\end{clm}
\begin{proof}
Note that $G^*[B]$ consists of some independent edges and isolated vertices by  Claim \ref{clm::4.2}. Then for any vertex $v \in B$, we obtain $d_B(v) \leq 1$, and so
\begin{equation}\label{equ::2.6}
|N_{G^*}(v) \cap(A \cup B)|=d_{A}(v)+d_{B}(v)\leq t+1.
\end{equation}
In addition, for any vertex $v\in R$, we know that $d_{A}(v)\leq t-1$ by the definition of $R$, and $d_{B}(v)\leq 2$ by Lemma \ref{lem::2.9}(i), thus
\begin{equation}\label{equ::2.4}
\left|N_{G^*}(v) \cap(A \cup B)\right|=d_{A}(v)+d_{B}(v) \leq t-1+2=t+1.
\end{equation}
Since $G^*[R]$ is $Q_{t}$-minor-free, we have
$2 e(R) \leq C_{1}|R|$ by Lemma \ref{lem::2.5}.
Moreover, by $A_{\alpha}(G^{*})\mathbf{x}=\rho_{\alpha}^{*}\mathbf{x}$ and (\ref{equ::2.1}), it is easy to see that
$$
\begin{aligned}
\alpha(n-1) \sum_{u \in R} x_{u} & \leq \rho_{\alpha}^* \sum_{u\in R} x_{u}=\sum_{u \in R}\left(\alpha d_{G^*}(u) x_{u}+(1-\alpha) \sum_{u v \in E(G^*)} x_{v}\right) \\
& \leq \sum_{u \in R}(\alpha d_{G^*}(u)+(1-\alpha) d_{G^*}(u))=\sum_{u \in R} d_{G^*}(u) \\
& \leq 2 e(R)+e(R, A\cup B)\leq C_{1}|R|+(t+1)|R|=(C_1+t+1)|R|,
\end{aligned}
$$
which implies that
\begin{equation}\label{equ::2.5}
\sum_{u \in R} x_{u} \leq \frac{(C_1+t+1)|R|}{\alpha(n-1)}.
\end{equation}
Now, take an arbitrary vertex $v \in V(G^*) \backslash A=B \cup R$. It follows from  (\ref{equ::2.6})-(\ref{equ::2.4}) that 
\begin{equation}\label{equ::2.8}
\left|N_{G^*}(v) \cap(A \cup B)\right| \leq t+1,
\end{equation}
and so
\begin{equation}\label{equ::2.7}
d_{G^*}(v) \leq \left|N_{G^*}(v) \cap(A \cup B)\right|+|R|\leq t+1+|R|.
\end{equation}
Furthermore, by $A_{\alpha}(G^{*})\mathbf{x}=\rho_{\alpha}^{*}\mathbf{x}$ and (\ref{equ::2.5})-(\ref{equ::2.8}), it is easy to see that
\begin{align*}
\left(\rho_{\alpha}^*-\alpha d_{_{G^*}}(v)\right) x_{v}&=(1-\alpha) \sum_{u v \in E(G^*)} x_{u} =(1-\alpha)\left(\sum_{\substack{u \in N_{G^*}(v) \\ u \in A \cup B}} x_{u}+\sum_{\substack{u \in N_{G^*}(v) \\ u \in R}} x_{u}\right) \\
&\leq(1-\alpha)\left(t+1+\sum_{u \in R} x_{u}\right)\leq(1-\alpha)\left(t+1+\frac{(C_1+t+1)|R|}{\alpha(n-1)}\right).
\end{align*}
Combining with (\ref{equ::2.1}), (\ref{equ::2.3}) and (\ref{equ::2.7}), we obtain
\begin{align*}
x_{v}  &\leq \frac{(1-\alpha)\left(t+1+\frac{(C_1+t+1)|R|}{\alpha(n-1)}\right)}{\rho_{\alpha}^*-\alpha d_{G^*}(v)} \leq \frac{(1-\alpha)\left(t+1+\frac{(C_1+t+1)|R|}{\alpha(n-1)}\right)}{\alpha(n-1)-\alpha(t+1+|R|)}\\
&\leq \frac{(1-\alpha)\left(t+1+\frac{(C_1+t+1)C^{\prime}\epsilon t n}{\alpha(n-1)}\right)}{\alpha(n-1)-\alpha(t+1+C^{\prime}\epsilon t n)}
\leq \frac{(1-\alpha)\cdot\frac{\alpha(t+1)+(C_1+t+1)C^{\prime}\epsilon t}{\alpha} \cdot \frac{n}{n-1}}{\alpha((1-C^{\prime}\epsilon t) n-t)}\\
&\leq\frac{(1-\alpha)\cdot\frac{\alpha(t+1)+(C_1+t+1)C^{\prime}\epsilon t}{\alpha} \cdot \frac{n}{n-1}}{\alpha(1-C^{\prime}\epsilon t-\epsilon) n}
 \leq \frac{\sqrt{2\alpha}}{2(C_{1}+2)},
\end{align*}
where the  second last inequality holds as long as $n$ is sufficiently large such that $n \geq \frac{t}{\epsilon}$, and the last inequality holds as long as $n$ is sufficiently large such that
$$
n \geq \frac{2(C_{1}+2)(1-\alpha)(\alpha(t+1)+(C_1+t+1)C^{\prime}\epsilon t)}{\sqrt{2\alpha}\alpha^2(1-C^{\prime}\epsilon t-\epsilon)}+1.
$$
\end{proof}
\begin{clm}\label{clm::4.4}
$R$ is empty.
\end{clm}
\begin{proof}
Suppose to the contrary that $R$ is not empty. Since $G^*[R]$ is  $Q_{t}$-minor-free, by Lemma \ref{lem::2.5} we have
$2e(R) \leq C_{1}|R|$, which implies that the average degree of $G^*[R]$ is at most $C_1$. Hence, there exists a vertex $v_{1} \in R$ such that $d_{R}(v_{1})\leq C_1$. By reusing Lemma \ref{lem::2.5}, we can order the vertices $v_2, \ldots, v_{|R|} \in R\setminus\{v_1\}$ such that $v_{i} \in R\setminus \{v_1, \ldots, v_{i-1}\}$ with $d_{R\setminus \{v_1, \ldots, v_{i-1}\}}(v_{i})\leq C_1$ for each $i=2, \ldots, |R| $. Clearly, $R=\{v_1, v_2, \ldots, v_{|R|}\}$, $v_1$ has at most $C_1$ neighbors in $R$, and each $v_i$ has at most $C_1$ neighbors in $R\setminus \{v_1, \ldots, v_{i-1}\}$ for $i \in \{2, \ldots, |R|\}$. By the definition of $R$, each $v_i$ ($i \in \{1, 2, \ldots, |R|\}$) has at least one non-neighbor in $A$. Besides,  it follows from Lemma \ref{lem::2.9}(i) that each vertex $v_i$ ($i \in \{1, 2, \ldots, |R|\}$) has at most two neighbors in $B$. Let 
\begin{align*}
G^{\prime}=& G^*-\{v_{i}u \in E(G^*) \mid  i\in \{1, 2, \dots, |R|\}, u \in R\}-\\
&\{ v_{i}w \in E(G^*)\mid  i\in \{1, 2, \dots, |R|\}, w \in B\}
+\{v_{i} z \notin E(G^*)\mid  i\in \{1, 2, \dots, |R|\}, z \in A\}.
\end{align*}
Clearly, $G^{\prime}=K_{t} \vee M_{n-t}$, thus $G^{\prime}$ is also $Q_t$-minor-free. Furthermore, by (\ref{equ::2.2}) and Claim \ref{clm::4.3} we obtain
\begin{align*}
\rho_\alpha\left(G^{\prime}\right)-\rho_\alpha^* \geq& \frac{\mathbf{x}^{T} A_\alpha\left(G^{\prime}\right) \mathbf{x}}{\mathbf{x}^{T} \mathbf{x}}-\frac{\mathbf{x}^{T} A_\alpha \left(G^{*}\right)\mathbf{x}}{\mathbf{x}^{T} \mathbf{x}} \\
=&\frac{1}{\mathbf{x}^{T} \mathbf{x}} \sum_{i=1}^{|R|} \left(\sum_{\substack{v_i z \notin E(G^*) \\
z \in A}} ( \alpha x_{v_i}^2+2(1-\alpha) x_{v_i} x_{z}+\alpha x_{z}^2) -\sum_{\substack{u v_i\in E(G^*) \\
u\in R}} ( \alpha x_{u}^2+ \right. \\
&\left.2(1-\alpha) x_{u} x_{v_i}+\alpha x_{v_i}^2)-\sum_{\substack{v_i w \in E(G^*) \\
w \in B}} ( \alpha x_{v_i}^2+2(1-\alpha) x_{v_i} x_{w}+\alpha x_{w}^2)\right) \\
\geq  &\frac{|R|}{\mathbf{x}^{T} \mathbf{x}}\left(\alpha(1-C^{\prime}\epsilon)^{2} -\frac{ (C_1+2)(\alpha+2(1-\alpha)+\alpha )(\sqrt{2\alpha})^2}{(2(C_1+2))^{2}}\right) \\
=&\frac{|R|\alpha}{\mathbf{x}^{T} \mathbf{x}}\left((1-C^{\prime}\epsilon)^{2} -\frac{1}{C_1+2}\right) >0,
\end{align*}
where the last inequality holds if we choose $\epsilon$ small enough such that $(1-C^{\prime}\epsilon)^{2} > \frac{1}{C_1+2}$. Hence, $\rho_\alpha\left(G^{\prime}\right)>\rho_\alpha^*$, a contradiction. 
\end{proof}
By Claims \ref{clm::4.2} and \ref{clm::4.4}, we get that $G^{*}$ is a subgraph of $K_{t} \vee M_{n-t}$. Then by Lemma \ref{lem::2.1-} and the maximality of  $\rho_{\alpha}^{*}$, we obtain $G^* = K_{t} \vee M_{n-t}$, as desired. This completes the proof of Theorem \ref{thm::1.2}.
\end{proof}

\section*{Declaration of competing interest}
The authors declare that they have no conflict of interest.

\section*{Data availability}
No data was used for the research described in the article.

\end{document}